\documentclass[preprint,12pt]{elsarticle}

\usepackage{amssymb}
\usepackage{amsmath, amsfonts, vmargin, enumerate}
\usepackage{verbatim}
\usepackage{amsthm}

\newtheorem{theorem}{Theorem}[section]
\newtheorem{corollary}{Corollary}[section]
\newtheorem{definition}{Definition}[section]
\newtheorem{lemma}{Lemma}[section]
\newtheorem{proposition}{Proposition}[section]

\numberwithin{equation}{section}

\newcommand{\A}{{\mathcal A}}

\newcommand{\E}{{\mathcal E}}

\newcommand{\M}{{\mathcal M}}

\newcommand{\D}{{\mathcal D}}

\newcommand{\8}{\infty}

\newcommand{\be}{\begin{eqnarray*}}
\newcommand{\ee}{\end{eqnarray*}}
\newcommand{\beq}{\begin{equation}}
\newcommand{\eeq}{\end{equation}}
\newcommand{\beqn}{\begin{equation*}}
\newcommand{\eeqn}{\end{equation*}}
\newcommand{\bsp}{\begin{split}}
\newcommand{\esp}{\end{split}}

\begin{document}

\begin{frontmatter}



\title{On  noncommutative weak Orlicz-Hardy spaces}

\author[mymainaddress]{Turdebek N. Bekjan\corref{mycorrespondingauthor}}
\ead{bekjant@yahoo.com}
\cortext[mycorrespondingauthor]{Corresponding author}
\author[mysecondaryaddress]{Madi Raikhan}

\ead{raikhan$_-$m@enu.kz}

\address[mymainaddress]{Faculty of Mechanics and Mathematics, L.N. Gumilyov Eurasian National University,  Nur-Sultan 010008, Kazakhstan.}
\address[mysecondaryaddress]{Astana IT University,  Nur-Sultan 010000, Kazakhstan}

\begin{abstract}

We introduce noncommutative  weak Orlicz spaces associated with a weight and study their properties. We also define noncommutative weak Orlicz-Hardy spaces and characterize their dual spaces.
\end{abstract}

\begin{keyword}
 noncommutative Lorentz space,  noncommutative Marcinkiewicz space,  weak noncommutative Orlicz space, noncommutative weak Orlicz-Hardy space.

 \MSC[2010] 46L52, 47L05.

\end{keyword}

\end{frontmatter}

\section{Introduction}
Al-Rashed and Zegarli$\rm\acute{n}$ski\cite{AZ} introduced  the noncommutative Orlicz spaces associated
to a normal faithful state on a semifinite von Neumann algebra. In \cite{ACA},  the authors considered a certain class of noncommutative Orlicz spaces,
associated with arbitrary faithful normal locally-finite weights on a semi-finite von Neumann algebra $\M$. In \cite{LHW}, the authors have investigated weak version of Orlicz spaces and proved the Burkholder-Gundy inequalities of martingales for this weak Orlicz spaces. The weak noncommutative Orlicz spaces were investigated in \cite{BCLJ} and it was used for the theory of noncom- mutative martingales. In this paper, we extend the results of \cite{ACA} to the weak noncommutative Orlicz space case.

The dual spaces of commutative weak $L_p$-spaces were characterized in \cite{Cw1,Cw2}, its noncommutative versions proved in \cite{C2,HS}. In \cite{C1,C2}, Ciach introduced  noncommutative Lorentz space and  noncommutative Marcinkiewicz space, and discussed  their dual spaces. The aim of this paper is to  define noncommutative weak Orlicz-Hardy spaces and characterize their dual spaces.

The paper is organized as follows.
In Section $2$ some necessary definitions and notations are collected including the  weak Orlicz spaces and the noncommutative weak Orlicz spaces. Using relationship between noncommutative weak Orlicz spaces and noncommutative Marcinkiewicz space, and Ciach's results to give dual spaces of  weak noncommutative Orlicz spaces. The noncommutative  weak Orlicz spaces associated with a weight are studied in Section $3$.
In Section $4$, we characterized  the dual spaces of noncommutative weak Orlicz-Hardy spaces.

\section{Preliminaries}

Let $\Omega= [0,\gamma)\;(0<\gamma\le\8)$  be equipped with the usual Lebesgue measure $\mu$. We denote by $L_0(\Omega)$ the space of $\mu$-measurable real-valued functions $f$ on $(\Omega)$ such that $\mu(\{\omega \in \Omega :\; |x(\omega)| > s\})<\8$ for some $s$. The decreasing rearrangement function $f^*: [0, \8) \mapsto [0, \8]$ for $f \in L_0 (\Omega)$ is defined by
$$
f^*(t) = \inf\{s > 0 : \; \mu ( \{\omega \in \Omega :\; |f (\omega)| > s\}) \le t\}
$$
for $t \ge 0$.

The  classical weak $L_p$-space $L_{p,\8}(\Omega)\;(0<p<\8)$ is defined as the set of all measurable
functions $f$ on $\Omega$ such that
\be
\|f\|_{Lp,\8}=\sup_{t>0}t^\frac{1}{p}f^*(t)<\8.
\ee
  However, for $p>1$
$L_{p,\8}(\Omega)$ can be renormed as a Banach space by
\be
f \mapsto \sup_{t >0} t^{-1+ \frac{1}{p}} \int_{0}^{t}f^*(s)d s.
\ee
We refer to \cite{G} for
more information about weak $L_p$-spaces.

A function $\Phi:(-\8,\8)\rightarrow [0,\8) $ is  called an N-function if it satisfies the following conditions:
(i) $\Phi$ is even and convex, (ii)
$\Phi(t)=0$~iff~$t=0$,
(iii)
$\lim_{t\rightarrow0}\frac{\Phi(t)}{t}=0,~\lim_{t\rightarrow
\infty}\frac{\Phi(t)}{t}=+\infty$.

Let $\phi(t)$ be the left derivative of $\Phi$. Then
$\phi(t)$  is left  continuous, nondecreasing on  $(0,\infty)$ and  satisfies:  $0<\phi(t)<\infty$
for  $0<t<\infty$,  $\phi(0)=0$ and $\lim_{t\rightarrow
\infty}\phi(t)=\infty$.  The left
inverse of $\phi$ ($\psi(s)=\inf\{t>0:\phi(t)>s\}$ for $s>0$) will be denoted by $\psi$. We define a complementary N-function $\Psi$ of  $\Phi$ by
\be
\Psi(|s|)=\int_{0}^{|s|}\psi(v)\,dv.
\ee
It is clear that $\Phi$ is the complementary N-function  of $\Psi$.
We call $(\Phi,\Psi)$ is a pair of complementary N-functions.

 Let $(\Phi,\Psi)$ be a pair of complementary N-functions, with inverses $\Phi^{-1},\;\Psi^{-1}$ (which are uniquely defined on $[0,\8)$). Then
 \beq\label{eq:inverses}
 t<\Phi^{-1}(t)\Psi^{-1}(t)<2t,\qquad t>0.
 \eeq

 An N-function $\Phi$ is said to satisfy the $\bigtriangleup_{2}$-condition  for all $t$, written  as $\Phi\in\bigtriangleup_{2}$, if there is $K>2$ such
 that $\Phi(2t)\leq k\Phi(t)$ for  all  $t\ge0 $. $\Phi$ is called  to satisfy the $\bigtriangledown_{2}$-condition
 for all $t$, written  as $\Phi\in\bigtriangledown_{2}$, if there is a constant $c>1$ such that $\Phi(t) \leq\frac{1}{2c}\Phi(ct)$ for all $t\ge0$.
For  a pair of complementary N-functions $(\Phi,\Psi)$, we have that $\Phi\in\bigtriangleup_{2}$ if and only if
$\Psi\in\bigtriangledown_{2}$ (see \cite[Theorem 2]{RR}).

 Let $(\Phi,\Psi)$ is a pair of complementary N-functions.  Then the Orlicz space on $\Omega$  associated with  $\Phi$ defined by
\be
L_{\Phi}(\Omega)=\bigg \{f\in
L_0(\Omega):\;\int_{0}^{\8}\Phi(|af(t)|)dt<\infty \; \mbox{for  some}\; a>0\bigg \}.
\ee
We define
\be
\|f\|_{\Phi}=\inf\bigg \{c>0:\;\int_{0}^{\8}\Phi(|\frac{f(t)}{c}|)dt\le1\bigg \}.
\ee
Then for any  $f\in L_{\Phi}(\Omega)$,
\be
\|f\|_{\Phi}\le\sup\bigg \{|\int_{0}^{\8}f(t)g(t)dt|: \int_{0}^{\8}\Psi(|g(t)|)dt\le1,\; g\in L^\Psi(\Omega)\bigg \}\le2\|f\|_{\Phi}.
\ee

For an N-function  $\Phi$, we define
\be
a_{\Phi} = \inf_{t>0} \frac{t \Phi'(t)}{\Phi (t)}\quad \text{and} \quad b_{\Phi} = \sup_{t>0} \frac{t \Phi'(t)}{\Phi (t)}.
\ee
Then  $1 \le a_{\Phi} \le  b_{\Phi} \le \8$ and $\Phi \in \triangle_2$  if and only if $b_{\Phi}< \8$. It is well-known that
\beq\label{eq:orlicz-dual}
L_{\Phi}(\Omega)^*=L_{\Psi}(\Omega),\qquad L_{\Phi}(\Omega)^*=L_{\Psi}(\Omega),
\eeq
with equivalent norms. We refer to \cite{RR} for the details on Orlicz spaces.

Now we consider the set of all measurable functions
$$
L_{\Phi,\8}(\Omega)=\{f\in
L_0(\Omega):\; \exists c>0, \Phi(\frac{t}{c})\mu(|f|>t)\leq1, \forall
t>0\}
$$
and denote
\be \|f\|_{\Phi,\8} = inf\{c>0:
\Phi(\frac{t}{c})\mu(|f|>t)\leq1, \forall t>0\}.
\ee
We call $L_{\Phi,\8}(\Omega)$  is a weak Orlicz space. If $\Phi(t)=t^p,$ then
$L_{\Phi,\8}(\Omega)=L_{p,\8}(\Omega)$ (see \cite{LHW} for more details).

Recall that
\be
L_{\Phi,\8}(\Omega) = \big \{ x \in L_{0}(\Omega):\; \exists\; c>0 \;\text{such that}\; \sup_{t > 0} t \Phi( f^*(t)/c ) < \8 \big \},
\ee
and
\be\begin{array}{rl}
\|x\|_{\Phi,\8} & = \inf \big \{ c>0:\; t \Phi(f^*(t)/c ) \leq 1, \forall t>0 \big \}\\
   & = \inf \big \{ c>0:\; \frac{1}{\Phi^{-1}(\frac{1}{t})} \mu_t (x)/c  \leq 1, \forall t>0 \big \}
   \end{array}
\ee
(see \cite[Proposition 3.1]{BCLJ}).

\subsection{Noncommutative weak $L_p$ spaces}\label{NcWeakLpSpa}

We keep all notations introduced in the above. In rest of this paper, $\Phi$ will always denote  an N-function and $\Psi$ denote a complementary N-function of  $\Phi$,   $\M$ always denote  a semifinite von Neumann algebra acting on a Hilbert space
$\mathbb{H}$ with a normal semifinite faithful trace $\tau\;(\tau(1)=\gamma)$.

For $0<p<\infty$ let $L_{p}(\M)$ denote the noncommutative $L_p$ space with respect to $(\M, \tau).$
As usual, we set $L_{\infty}(\M,\tau)=\M$ equipped with the operator norm. Also, let $L_{0}(\mathcal{M})$ denote the topological $*$-algebra of
measurable operators with respect to $(\M, \tau).$

For $x\in L_{0}(\M)$, we define
\be
\lambda_{s}(x)=\tau(e^{\perp}_s (|x|))\;(s>0)\; \; \text{and}\; \; \mu_t (x) = \inf \{ s>0:\;
\lambda_s (x) \le t \}\; (t >0),
\ee
where $e_s^{\perp} (|x|) = e_{(s,\infty)}(|x|)$ is the spectral projection of $|x|$ associated with the interval $(s,\infty)$. We call the function $s\mapsto\lambda_{s}(x)$ is the distribution function of $x$ and $\mu_{t}(x)$  is the generalized singular number of $x$. For simplicity, we denote  by $\lambda (x)$ and  $\mu(x)$ the two functions $s \mapsto \lambda_s (x)$ and $t \mapsto \mu_t (x),$ respectively. It is clear that both functions $\lambda (x)$ and $\mu(x)$ are decreasing and continuous from the right on $(0,\infty)$ (for further information, see  \cite{FK}).


For $0< p < \8,$ the noncommutative weak $L_p$ space $L_{p,\8}(\mathcal{M})$ is defined as the space
of all measurable operators $x$ such that
\be\begin{split}
\|x\|_{p,\8}= \sup_{t > 0} t^{\frac{1}{p}}\mu_{t}(x) < \8.
\end{split}\ee
Equipped with $\|.\|_{L_{p,\8}},$ $L_{p,\8}(\mathcal{M})$ is a quasi-Banach space. However, for $p>1$
$L_{p,\8}(\mathcal{M})$ can be renormed as a Banach space by
\be
x \mapsto \sup_{t >0} t^{-1+ \frac{1}{p}} \int_{0}^{t}\mu_{s}(x)d s.
\ee
On the other hand, the quasi-norm admits the following useful description
\beq\label{eq:WeakLpNormMu}
\|x\|_{p,\8} = \inf \big \{c>0:\; t ( \mu_t (x)/c )^p \le 1,\; \forall t >0 \big \}.
\eeq
Also, we have a description in terms of distribution function as follows
\beq\label{eq:WeakLpNormLmbda}
\|x\|_{p,\8} = \sup_{s > 0} s \lambda_s (x)^{\frac{1}{p}}.
\eeq

Recall that noncommutative weak $L_p$ spaces can be presented through noncommutative Lorenz spaces, for details see  \cite{DDP2} and Xu \cite{X}.

\subsection{Noncommutative  weak Orlicz spaces }

Let
\be
L_\Phi(\M)=\{x\in
L_0(\M):\tau(\Phi(|x|))=\int_0^{\tau(1)}\Phi(\mu_t(x))dt<\infty\}
\ee
and
$$
\|x\|_{\Phi}=\inf\{\lambda>0:\tau(\Phi(\frac{|x|}{\lambda}))\le1\}, \quad  \forall x\in
L_\Phi(\M).
$$
Then $L_\Phi(\M)$ is a Banach  space. We call it is
the noncommutative Orlicz space on $(\M,\tau)$.

\begin{definition}\label{weakOrlicz spaces} The noncommutative weak Orlicz space $L_{\Phi,\8} (\M)$ is defined as following:
\be
L_{\Phi,\8} (\M) = \left\{ x \in L_{0}(\mathcal{M}):\;  \sup_{t > 0}
t \Phi( \mu_t (x)) < \8 \right\},
\ee
equipped with
\be
\|x\|_{\Phi,\8} = \inf \left \{c>0:\; t \Phi( \mu_t (x)/c ) \leq 1,\; \forall t>0 \right\}.
\ee
\end{definition}
If $\Phi(t)=t^p$ with $1\le p < \8$,  then $L_{\Phi,\8}(\M)$ is the noncommutative weak $L_p$-space.

Recall that $L_{\Phi,\8} (\M)$  is a quasi-Banach space, and for any $x\in L_{\Phi,\8} (\M)$
\be
\|x\|_{\Phi,\8} = \inf \left\{c>0:\; \frac{1}{\Phi^{-1}(\frac{1}{t})} \mu_t (x)/c  \leq 1,\; \forall t>0 \right\}=\sup_{t>0}\frac{1}{\Phi^{-1}(\frac{1}{t})} \mu_t (x).
\ee
For any $c>0$ we have that
\beq\label{eq:DistrSingularEquiv}
\sup_{t >0} t \Phi (\mu_t (x)/c) =\sup_{s > 0} \lambda_s (x) \Phi (s /c),\;\forall x \in L_0 (\M).
\eeq
For more information on noncommutative weak Orlicz spaces, see \cite{BCLJ}.

For any $x\in L_0(\M)$,  set $\tilde{\mu}_t(x)=\int_0^t\mu_s(x)ds$. Then $ \mu_t(x)\le\tilde{\mu}_t(x)$ for all $t>0$ and the map $x\mapsto\tilde{\mu}(x)$ is a sublinear operator from
$L_0(\M)$ to $L_0(\Omega)\;(a=\tau(1))$.

\begin{proposition}\label{pro:equivalent norm}
 If $\Phi$ is an Orlicz function with $1< a_{\Phi} \le b_{\Phi}<\8$, then there exists a constant $C>0$ such that
\beq\label{eq:mu-t-mu*-t-1}
\sup_{t>0}t \Phi \big [ \tilde{\mu}_t (x) \big ] \leq C \sup_{t>0} t \Phi \big [ \mu_t (x) \big ]
\eeq
for all $x \in L_{\Phi,\8}(\M)$. Consequently,
\beq\label{eq:mu-t-mu^*-t-2}
\sup_{t>0}\frac{1}{\Phi^{-1}(\frac{1}{t})} \tilde{\mu}_t (x) \le C' \sup_{t>0}\frac{1}{\Phi^{-1}(\frac{1}{t})} \mu_t (x),\quad \forall x \in L^{\Phi,\8}(\M).
\eeq
\end{proposition}
\begin{proof} Let $1<p_0< a_{\Phi} \le b_{\Phi}<p_1<\8$. By \cite[Theorem {\rm III}.3.8 and {\rm III}.3.10]{BS}, the map $x\mapsto\tilde{\mu}(x)$ is bounded  from
$L_{p_i}(\M)$ to $L_{p_i}(\Omega)$, $i=0,\;1$. Using \cite[Corollary 4.4]{BCLJ}, we obtain the desired result.

\end{proof}
We use \eqref{eq:inverses} and the above proposition to obtain the following result.
\begin{corollary}\label{cor:equivalent norm}
 Let $\Phi$ be an Orlicz function with $1< a_{\Phi} \le b_{\Phi}<\8$. Set
 \be
\|x\|_{(\Phi)',\8}=\sup_{t>0}\Psi^{-1}(\frac{1}{t})\int_0^t\mu_s(x)ds,\qquad \forall x \in L^{\Phi,\8}(\M).
\ee
Then $\|x\|_{(\Phi)',\8}$ is an equivalent norm on $L_{\Phi,\8}(\M)$.

\end{corollary}

Set $\varphi(t)=1/\Psi^{-1}(\frac{1}{t})$. Then $\varphi$ is
an increasing concave function on $(0,\8)$ with $\lim_{t\rightarrow0}\varphi(t)=0$ and $\lim_{t\rightarrow\8}\varphi(t)=\8$.
Let $\Lambda_\varphi(\Omega),\;M_\varphi(\Omega)$ be the usual Lorentz
and Marcinkiewicz spaces with norms defined by
\be
\Lambda_\varphi(\Omega)=\{f\in L_0(\Omega):\;\|f\|_{\Lambda_\varphi}=\int_{0}^{\8}f^*(t)\varphi'(t)dt<\8\}
\ee
and
\be
M_\varphi(\Omega)=\{f\in L_0(\Omega):\;\|f\|_{M_\varphi}=\sup_{t>0}\frac{1}{\varphi(t)}\int_0^tf^*(t)ds<\8\}.
\ee
The Lorentz space
 $\Lambda_\varphi(\Omega)$ has order continuous norm and  $\Lambda_\varphi(\Omega)^*=M_\varphi(\Omega)$. If  $M_\varphi^0(\Omega)$ denotes the linear subspace of  $M_\varphi(\Omega)$ consisting of all  $f\in M_\varphi(\Omega)$ for which
 \be
 \limsup_{t\rightarrow0}\frac{1}{\varphi(t)}\int_0^t\mu_s(x)ds=0\quad\mbox{and}\quad  \limsup_{t\rightarrow\8}\frac{1}{\varphi(t)}\int_0^t\mu_s(x)ds=0.
 \ee
Then $M_\varphi^0(\Omega)^*=\Lambda_\varphi(\Omega)$ (for more details see \cite[Chapter {\rm II}.5]{KPS} . Since  $M_\varphi^0(\Omega)$ is separable,
\be
M_\varphi^0(\M)=\mbox{closure
 of}\;S(\M)\; \mbox{in
 }\; M_\varphi(\M).
\ee
Hence,
\be
\Lambda_\varphi(\M)^*=M_\varphi(\M),\qquad M_\varphi^0(\M)^*=\Lambda_\varphi(\M)
\ee
(see \cite[Proposition 5.3]{DDP2}, also see \cite[Theorem 2.1]{C1} and \cite[Proposition 2.1]{C2}).

 Let  $1< a_{\Phi} \le b_{\Phi}<\8$. By Proposition \ref{pro:equivalent norm},
 \be
 M_\varphi(\M)=L_{\Phi,\8}(\M).
\ee
We denote the closure
 of $S(\M)$ in $L_{\Phi,\8}(\M)$ by $L_{\Phi,\8}^0(\M)$ and $\Lambda_\varphi(\M)$ by $L_{1,\Psi}(\M)$, respectively. Then
 \beq\label{eq:weak-dual}
 L_{1,\Psi}(\M)^*=L_{\Phi,\8}(\M)\quad\mbox{and}\quad L_{\Phi,\8}^0(\M)^*=L_{1,\Psi}(\M)
 \eeq
It is clear that
\be
\lim_{t\rightarrow0}\frac{t}{\varphi(t)}=\lim_{t\rightarrow0}t\Psi^{-1}(\frac{1}{t})=0,\quad \lim_{t\rightarrow\8}\frac{t}{\varphi(t)}=\lim_{t\rightarrow\8}t\Psi^{-1}(\frac{1}{t})=\8.
\ee
We define continuous seminorms $N_0$ and $N_\8$ on $L_{\Phi,\8}(\M)$ by
 \be
N_0(x)=\limsup_{t\rightarrow0}\frac{1}{\varphi(t)}\int_0^t\mu_s(x)ds
\ee
and
\be N_\8(x)=\limsup_{t\rightarrow\8}\frac{1}{\varphi(t)}\int_0^t\mu_s(x)ds,
 \ee
for all $x\in L_{\Phi,\8}(\M)$.
Using main result in \cite{C2}, we obtain that
\beq\label{eq:dual-weak-orlicz}
 L_{\Phi,\8}(\M)^*=L_{1,\Psi}(\M)\oplus S_0\oplus S_\8,
\eeq
where
\be
\begin{array}{rl}
     S_0 &=\{\ell\in L_{\Phi,\8}(\M)^*:\;\ell\;\mbox{ annihilates all}\; x\in\M\} \\
      &=\{\ell\in L_{\Phi,\8}(\M)^*:\exists\: C>0,\;
|\ell(x)|\le CN_0(x),\quad\forall x\in L_{\Phi,\8}(\M)\},
\end{array}
\ee
\be
\begin{array}{rl}
     S_\8 &=\{\ell\in L_{\Phi,\8}(\M)^*:\;\ell\;\mbox{ annihilates all}\;  x\in L_{\Phi,\8}(\M)\;\mbox{with}\; r(x)\in S(\M)\} \\
      &=\{\ell\in L_{\Phi,\8}(\M)^*:\exists\: C>0,\;
|\ell(x)|\le CN_\8(x),\quad\forall x\in L_{\Phi,\8}(\M)\}.
\end{array}
\ee

Let
\be M(t, \varphi)= \sup_{s >0} \frac{\varphi(t s)}{\varphi(s)},\quad t >0.
\ee
Define
\be
p_{\varphi} = \lim_{t \searrow 0} \frac{\log M(t, \varphi)}{\log t}, \quad q_{\varphi} = \lim_{t \nearrow \8} \frac{\log M(t, \varphi)}{\log t}.
\ee
Then
\beq\label{eq:weak index-strong index}
[p_{\varphi},q_{\varphi} ]\subset[\frac{1}{b_\Phi},\frac{1}{a_\Phi}]
\eeq
(For more details, see \cite[Remarks 3 (p.84) and  Theorem 11.11]{Ma2} and \cite[Theorem 4.2]{Ma1}).

\section{Noncommutative  weak Orlicz spaces associated with a weight}

We denote by $L_{loc}(\M)$ by the set of all measurable
locally-measurable  operators affiliated with $\M$. It is well-known that $L_{loc}(\M)$ is a $\ast$-algebra with respect to the strong sum and strong product and
$L_{0}(\M)$  is a $*$-subalgebra in $L_{loc}(\M)$ (see \cite{MC1,MC2}). Set $M^+=\{x\in\M:\;x\ge0\}$ and $L_{loc}(\M)^+=\{x\in L_{loc}(\M):\;x\ge0\}$. Let
\be
\tilde{\tau}(x)=\sup\{\tau(y):\;y\in \M^+,\;y\le x\}, \qquad x\in L_{loc}(\M)^+.
\ee
Then $\tilde{\tau}$ is an extension of $\tau$ to $L_{loc}(\M)^+$ (see \cite[$\S$4.1]{MC2}). The extension  will be denoted still by $\tau$.

\begin{definition}
 \begin{enumerate}[\rm(i)]
 \item A {\it   weight} on $\M$ is a map $\omega: \M^+\to [0, \infty]$
satisfying
 \be
 \omega(x+\lambda y)=\omega(x)+\lambda\omega(y),\qquad\forall\;x, y\in \M^+,\ \forall\; \lambda\in\mathbb{R}
\ee
(where $0.\8=0$).
\item A  weight $\omega$ is said to be {\it normal} if
$\sup_i\omega(x_i)=\omega(\sup_i x_i)$ for any bounded increasing net
$(x_i)$ in $\M^+$, {\it faithful} if $\omega(x)=0$ implies $x=0$, {\it
semifinite} if the linear span $\M_\omega$ of the cone $\M_\omega^+=\{x \in\M^+:\; \omega(x)<\8\}$ is
dense in $\M$ with respect to the ultra-weak topology, and {\it locally finite} if for any non-zero
$x\in\M^+$ there is a non-zero $y\in\M^+$ such that $y\le x$ and
$0<\omega(y)<\infty$.
\end{enumerate}
\end{definition}
Let $\omega$  be a faithful normal semifinite weight on $\M$. Then  $\omega$  has a Radon-Nikodym derivative $D_{\omega}$ with respect to
$\tau$ such that $\omega(\cdot)=\tau(D_{\omega}\cdot)$ (see \cite{PT}).
The weight $\omega$ is locally finite if and only if the operator  $D_{\omega}$  is locally measurable (see \cite{Tr}). In the sequel, unless otherwise specified, we always denote by  $\omega$  a faithful normal locally finite weight on $\M$. Let $\Phi^{-1}:[0,\8)\rightarrow[0,\8)$ be the inverse of $\Phi$  (which is uniquely defined on $\mathbb{R}^+$).

Let
\be
\M_{\Phi,\8}^{\alpha,\omega}=\bigg \{x\in\M:\;\sup_{t>0}t \mu_t(\Phi\big(|\Phi^{-1}(D_{\omega})^\alpha x\Phi^{-1}(D_{\omega})^{1-\alpha}|))<\8\bigg \}
\ee
and
\be
\|x\|_{\Phi,\8,\alpha,\omega}=\inf\bigg \{c>0:\;\sup_{t>0}t \Phi( \mu_t (\Phi^{-1}(D_{\omega})^\alpha x\Phi^{-1}(D_{\omega})^{1-\alpha})/c)\le1\bigg \}.
\ee

\begin{lemma}
$\M_{\Phi,\8}^{\alpha,\omega}$ is linear subspace in $\M$.
\end{lemma}
\begin{proof} Let $x\in \M_{\Phi,\8}^{\alpha,\omega}$ and $\eta\in\mathbb{C}$. If $|\eta|\le1$, by Lemma 2.5 in \cite{FK} and convexity of $\Phi$,
\be
\begin{array}{rl}
&\sup_{t>0}t \mu_t(\Phi\big(|\Phi^{-1}(D_{\omega})^\alpha \eta x\Phi^{-1}(D_{\omega})^{1-\alpha}|\big))\\
&\qquad=\sup_{t>0}t \Phi\big( \mu_t(\Phi^{-1}(D_{\omega})^\alpha \eta x\Phi^{-1}(D_{\omega})^{1-\alpha})\big)\\
&\qquad=\sup_{t>0}t \Phi\big(|\eta|\mu_t (\Phi^{-1}(D_{\omega})^\alpha x\Phi^{-1}(D_{\omega})^{1-\alpha})\big)\\
&\qquad\le|\eta|\sup_{t>0}t \Phi\big(\mu_t (\Phi^{-1}(D_{\omega})^\alpha x\Phi^{-1}(D_{\omega})^{1-\alpha})\big)\\
&\qquad=|\eta|\sup_{t>0}t \mu_t(\Phi\big(|\Phi^{-1}(D_{\omega})^\alpha x\Phi^{-1}(D_{\omega})^{1-\alpha}|\big))<\8.
\end{array}
\ee
Hence, $\eta x\in \M_{\Phi,\8}^{\alpha,\omega}$.  If $|\eta|>1$, since $\Phi\in\bigtriangleup_{2}$, there exists a constant $k=k(|\eta|)>0$ such that $\Phi(|\eta|t)\le k\Phi(t)$ for all
$t>0$. Similar to the above, we obtain that $\eta x\in \M_{\Phi,\8}^{\alpha,\omega}$.

Now let  $x,y\in \M_{\Phi,\8}^{\alpha,\omega}$. Using  Lemma 2.5 in \cite{FK}, convexity of $\Phi$ and $\Phi\in\bigtriangleup_{2}$, we get
\be
\begin{array}{rl}
&\sup_{t>0}t \mu_t(\Phi\big(|\Phi^{-1}(D_{\omega})^\alpha x+y\Phi^{-1}(D_{\omega})^{1-\alpha}|\big))\\
&\qquad=\sup_{t>0}t \Phi\big( \mu_t(\Phi^{-1}(D_{\omega})^\alpha x+y\Phi^{-1}(D_{\omega})^{1-\alpha})\big)\\
&\qquad\le\sup_{t>0}t \Phi\big( \mu_{t/2}(\Phi^{-1}(D_{\omega})^\alpha x\Phi^{-1}(D_{\omega})^{1-\alpha})\\
&\qquad\qquad\qquad\qquad\qquad\qquad+\mu_{t/2}(\Phi^{-1}(D_{\omega})^\alpha y\Phi^{-1}(D_{\omega})^{1-\alpha})\big)\\
&\qquad\le c\sup_{t>0}t/2\Phi\big(\mu_{t/2}(\Phi^{-1}(D_{\omega})^\alpha x\Phi^{-1}(D_{\omega})^{1-\alpha})\big)\\
&\qquad\qquad\quad+ c\sup_{t>0}t/2\Phi\big(\mu_{t/2}(\Phi^{-1}(D_{\omega})^\alpha y\Phi^{-1}(D_{\omega})^{1-\alpha})\big)<\8,
\end{array}
\ee
and so $x+y\in \M_{\Phi,\8}^{\alpha,\omega}$.
\end{proof}
Similar to \cite[Proposition 3.2]{BCLJ}, we have the following result.
\begin{proposition}\label{prop:ineq}
\begin{enumerate}[\rm (i)]

\item If $\|x\|_{\Phi,\8,\alpha,\omega} > 0$ then
\be
\sup_{t>0} t \Phi \big ( \mu_t (\Phi^{-1}(D_{\omega})^\alpha x\Phi^{-1}(D_{\omega})^{1-\alpha})/\|x\|_{\Phi,\8,\alpha,\omega} \big ) \leq 1.
\ee
\item $\|x\|_{\Phi,\8,\alpha,\omega}$ is a quasi-norm on the linear space $\M_{\Phi,\8}^{\alpha,\omega}$ and
\beq\label{eq:QuasiNormInequa}
\| x + y \|_{\Phi,\8,\alpha,\omega} \le 2 ( \|x\|_{\Phi,\8,\alpha,\omega} + \|y\|_{\Phi,\8,\alpha,\omega}),\; \forall x, y \in \M_{\Phi,\8}^{\alpha,\omega}.
\eeq
\item If $\|x\|_{\Phi,\8,\alpha,\omega}\leq 1,$ then
\be
\sup_{t>0} t \Phi(\mu_t (\Phi^{-1}(D_{\omega})^\alpha x\Phi^{-1}(D_{\omega})^{1-\alpha})) \leq \|x\|_{\Phi,\8,\alpha,\omega}.
\ee

\item $\| x\|_{\Phi,\8,\alpha,\omega} \le \| x \|_{\Phi,\alpha,\omega}$ for any $x \in \M_{\Phi}^{\alpha,\omega}$,
where
$$
\M_{\Phi}^{\alpha,\omega}= \{x\in\M:\;\Phi^{-1}(D_{\omega})^\alpha x\Phi^{-1}(D_{\omega})^{1-\alpha}\in
L_{\Phi} (\M) \}
$$
and
$\| x \|_{\Phi,\alpha,\omega}=\|\Phi^{-1}(D_{\omega})^\alpha x\Phi^{-1}(D_{\omega})^{1-\alpha}\|_{\Phi}$. Consequently, $\M_{\Phi}^{\alpha,\omega}\subset \M_{\Phi,\8}^{\alpha,\omega}$.

\end{enumerate}
\end{proposition}

\begin{definition} Let $\omega$  be a faithful normal semifinite weight on $\M$ and $\alpha\in[0,1]$. We call the completion of $(\M_{\Phi,\8}^{\alpha,\omega}, \|\cdot\|_{\Phi,\8,\alpha,\omega})$ is  the weak noncommutative Orlicz space associted with $\Phi,\M$ and $\omega$, denote by $L_{\Phi,\8}^{\alpha,\omega}(\M,\tau)$.
\end{definition}

\begin{lemma}\label{lem:density}
Let $D$ be a positive nonsingular operator  in $L^1(\M)$. If $\alpha\in[0,1]$, then $\Phi^{-1}(D)^{\alpha}\M\Phi^{-1}(D)^{1-\alpha}$ is dense in $L_{\Phi,\8}^0(\M)$.
\end{lemma}
\begin{proof} Set $e_n=e_{(\frac{1}{n},n]}(D)$, for any $n\in\mathbb{N}$. Then $e_n$ increases strongly to $1$ and $\tau(e_n)<\8$, for any $n\in\mathbb{N}$. Let $x\in S(\M)$. Then there is a projection $e$ in $\M$ such that $\tau(e)<\8$ and $ex=xe=x$. Hence, $x\in L^1(\M)$. By \cite[Lemma 2.1]{J}, we get
 $\lim_{n\rightarrow\8}\|xe_{n}-x\|_{1}=0$. It follows that  $xe_{n}- x\rightarrow0$ in measure as $n\rightarrow\8$.
Using \cite[Lemma 3.1]{FK}, we get for any $t>0$, $\mu_t(xe_{n}- x)\rightarrow0$  as $n\rightarrow\8$. On the other hand, by Lemma 2.5 in \cite{FK},  $\mu_t(xe_{n}- x)\le\mu_t(x)$ for all $t>0$. Applying Lebesgue dominated convergence theorem, we get
\be
\lim_{n\rightarrow\8}\tau(\Phi(|xe_n-x|))=\lim_{n\rightarrow\8}\int_{0}^{\tau(1)}\Phi(\mu_t(xe_n-x)dt=0.
\ee
Therefore, $\lim_{n\rightarrow\8}\|xe_n-x\|_{\Phi}=0$. Similarly,
$\lim_{n\rightarrow\8}\|e_n x-x\|_{\Phi}=0$. Using (iv) of Proposition \ref{prop:ineq}, we obtain that $\lim_{n\rightarrow\8}\|xe_n-x\|_{\Phi,\8}=0$ and
$\lim_{n\rightarrow\8}\|e_n x-x\|_{\Phi,\8}=0$, and so
$$
\lim_{n\rightarrow\8}\|e_nxe_n-x\|_{\Phi,\8}\leq2[\lim_{n\rightarrow\8}\|xe_n-x\|_{\Phi,\8}+\lim_{n\rightarrow\8}\|e_{n}x-x\|_{\Phi,\8}]=0,
$$
i.e., the closure of $\cup_{n=1}^{\8}e_n\M e_n$ in $L_{\Phi,\8}(\M)$ contains $S(\M)$. Thus, $\cup_{n=1}^{\8}e_n\M e_n$ is dense in $L_{\Phi,\8}^0(\M)$.

Next, we prove that $\Phi^{-1}(D)^{\alpha}\M\Phi^{-1}(D)^{1-\alpha}\subset L_{\Phi,\8}^0(\M)$. Set $\Phi^{(p)}(t) =\Phi(t^p)$, for $1<p<\8$. Let $y\in\M$. If $\alpha\in(0,1)$, then
\be
\lim_{n\rightarrow\8}\tau(\Phi^{(\frac{1}{\alpha})}(|\Phi^{-1}(D)^{\alpha}-e_n\Phi^{-1}(D)^{\alpha}|)=\lim_{n\rightarrow\8}\tau(D-De_n)=0.
\ee
It follows that $\lim_{n\rightarrow\8}\|\Phi^{-1}(D)^{\alpha}-e_n\Phi^{-1}(D)^{\alpha}\|_{\Phi^{(\frac{1}{\alpha})},\8}=0$. Since $\Phi^{-1}(\frac{s}{2})\ge\frac{1}{2}\Phi^{-1}(s)$ for all $s>0$, by Lemma 2.5 in \cite{FK}, we get
\be
\begin{array}{rl}
&\|\Phi^{-1}(D)^{\alpha}y\Phi^{-1}(D)^{1-\alpha}-e_n\Phi^{-1}(D)^{\alpha}y\Phi^{-1}(D)^{1-\alpha}\|_{\Phi,\8}\\
&\qquad=\sup_{t>0}\frac{1}{\Phi^{-1}(\frac{1}{2t})} \mu_{2t} (\Phi^{-1}(D)^{\alpha}y\Phi^{-1}(D)^{1-\alpha}-e_n\Phi^{-1}(D)^{\alpha}y\Phi^{-1}(D)^{1-\alpha})\\
&\qquad\le2\|y\|\sup_{t>0}\frac{1}{\Phi^{-1}(\frac{1}{t})} \mu_{t} ((1-e_n)\Phi^{-1}(D)^{\alpha})\mu_{t} (\Phi^{-1}(D)^{1-\alpha})\\
&\qquad\le2\|y\|\|(1-e_n)\Phi^{-1}(D)^{\alpha}\|_{\Phi^{(\frac{1}{\alpha})},\8}\|\Phi^{-1}(D)^{1-\alpha}\|_{\Phi^{(\frac{1}{1-\alpha})},\8}.
\end{array}
\ee
Hence, $\lim_{n\rightarrow\8}\|\Phi^{-1}(D)^{\alpha}y\Phi^{-1}(D)^{1-\alpha}-e_n\Phi^{-1}(D)^{\alpha}y\Phi^{-1}(D)^{1-\alpha}\|_{\Phi,\8}=0$. Similar to the above, $\lim_{n\rightarrow\8}\|\Phi^{-1}(D)^{\alpha}y\Phi^{-1}(D)^{1-\alpha}-e_n\Phi^{-1}(D)^{\alpha}y\Phi^{-1}(D)^{1-\alpha}e_n\|_{\Phi,\8}=0$. On the other hand,
$e_n\Phi^{-1}(D)^{\alpha},\;\Phi^{-1}(D)^{1-\alpha}e_n\in\M$, and so for all $n\in\mathbb{N}$,
\be
e_n\Phi^{-1}(D)^{\alpha}y\Phi^{-1}(D)^{1-\alpha}e_n\in e_n\M e_n.
\ee
 Therefore, $\Phi^{-1}(D)^{\alpha}y\Phi^{-1}(D)^{1-\alpha}\in L_{\Phi,\8}^0(\M)$. In the case $\alpha=0,1$, this result also holds. Thus $\Phi^{-1}(D)^{\alpha}\M\Phi^{-1}(D)^{1-\alpha}\subset L_{\Phi,\8}^0(\M)$.

Finally, we prove that $\Phi^{-1}(D)^{\alpha}\M\Phi^{-1}(D)^{1-\alpha}$ is dense in $L_{\Phi,\8}^0(\M)$. For this it is sufficient to prove that $\cup_{n=1}^{\8}e_n\M e_n\subset\Phi^{-1}(D)^{\alpha}\M\Phi^{-1}(D)^{1-\alpha}$. Since for any $ n\in\mathbb{N}$, $e_n\Phi^{-1}(D)^{-\alpha},\;\Phi^{-1}(D)^{\alpha-1}e_n\in\M$, we have that
\be
\begin{array}{rl}
e_n\M e_n&=\Phi^{-1}(D)^{\alpha}\Phi^{-1}(D)^{-\alpha}e_n\M e_n\Phi^{-1}(D)^{\alpha-1}\Phi^{-1}(D)^{1-\alpha}\\
&=\Phi^{-1}(D)^{\alpha}e_n\Phi^{-1}(D)^{-\alpha}\M \Phi^{-1}(D)^{\alpha-1}e_n\Phi^{-1}(D)^{1-\alpha}\\
&\subset\Phi^{-1}(D)^{\alpha}\M\Phi^{-1}(D)^{1-\alpha}.
\end{array}
\ee
It follows that $\cup_{n=1}^{\8}e_n\M e_n\subset\Phi^{-1}(D)^{\alpha}\M\Phi^{-1}(D)^{1-\alpha}$.
\end{proof}

\begin{theorem}\label{thm:eqivalent-weakspace} Let $\omega$  be a faithful normal semifinite weight on $\M$ such that its the Radon-Nikodym derivative $D_{\omega}$ with respect to
satisfy $D_{\omega}\in L^1(\M)$.
 If $\alpha\in[0,1]$, then $L_{\Phi,\8}^{\alpha,\omega}(\M,\tau)$ and  $L_{\Phi,\8}^0(\M)$  are isometrically isomorphic.
\end{theorem}
\begin{proof} We define $T:\M_{\Phi,\8}^{\alpha,\omega}\rightarrow L_{\Phi,\8}^{\alpha,\omega}(\M,\tau)$ by
\be
T(x)=\Phi^{-1}(D)^{\alpha}x\Phi^{-1}(D)^{1-\alpha},\qquad x\in\M_{\Phi,\8}^{\alpha,\omega}.
\ee
Then $T$ is a linear isometry from $\M_{\Phi,\8}^{\alpha,\omega}$ to $\Phi^{-1}(D)^{\alpha}\M\Phi^{-1}(D)^{1-\alpha}$. By the definition of $L_{\Phi,\8}^{\alpha,\omega}(\M,\tau)$ and Lemma \ref{lem:density}, we know that $\M_{\Phi,\8}^{\alpha,\omega}$ is dense in $L_{\Phi,\8}^{\alpha,\omega}(\M,\tau)$ and  $\Phi^{-1}(D)^{\alpha}\M\Phi^{-1}(D)^{1-\alpha}$ is dense in  $L_{\Phi,\8}^0(\M)$. Hence, we can extend to an isometric isomorphism between $L_{\Phi,\8}^{\alpha,\omega}(\M,\tau)$ and  $L_{\Phi,\8}^0(\M)$.
\end{proof}

\section{Noncommutative  weak Orlicz-Hardy spaces }

We will assume that $\mathcal{D}$ is a von Neumann
 subalgebra of $\mathcal{M}$ such that the restriction of $\tau$ to $\mathcal{D}$ is still semifinite. Let $\mathcal{E}$ be the (unique) normal faithful
 conditional expectation of $\mathcal{M}$ with respect to $\mathcal{D}$ which
 leaves $\tau$ invariant.

\begin{definition} A w*-closed subalgebra $\A$ of $\M$ is
called a subdiagonal subalgebra of $\M$ with respect to $\E$(or
$\D$) if \begin{enumerate}[\rm (i)]

\item $\A+ J(A)$ is w*-dense in  $\M$, where $J(A)=\{x^*:\;x\in\A\}$,

\item $\E(xy)=\E(x)\E(y),\; \forall\;x,y\in
\A$,

\item $\A\cap J(A)=\D$.
\end{enumerate}
$\D$ is then called the diagonal of $\A$.
\end{definition}

In this section  $\M$ always denotes a semifinite von Neumann algebra  with a normal semi-finite faithful trace $\tau$ satisfying $\tau(1)=\gamma$ and $\A$ denotes
 a subdiagonal subalgebra of $\M$ with respect to $\E$(or
$\D$). We keep all notations introduced in the previous  section.

Let
\be
H_{\Phi,\8}(\A)=\{x\in L_{\Phi,\8}(\M):\quad \tau(xa)=0,\quad \forall a\in\A_0\}.
\ee
Then $H_{\Phi,\8}(\A)$ is called noncommutative weak Orlicz-Hardy space associated with $\mathcal{A}$. Similarly, we define
$H_{1,\Psi}(\A)$ by
\be
H_{1,\Psi}(\A)=\{x\in L_{1,\Psi}(\M):\quad \tau(xa)=0,\quad \forall a\in\A_0\}.
\ee

Let $\M$ be finite ($\tau(1)=\alpha<\8$). By Propositions 4.3 in \cite{Be}, we know that $H_{1,\Psi}(\A)$ is the closure of $\A$ in
$L_{1,\Psi}(\M)$.

\begin{proposition}\label{pro:weak-hardy space}
 Let $1< a_{\Phi} \le b_{\Phi}<\8$. Then
\be
L_{1,\Psi}(\M)=H_{1,\Psi}(\A)\oplus J(H_{1,\Psi}(\A)_0),
\ee
where $H_{1,\Psi}(\A)_0=\{x\in H_{1,\Psi}(\A):\;\E(x)=0\}$.
\end{proposition}
\begin{proof} Since the lower Boyd index $p_{L_{1,\Psi}}$  and upper Boyd index $q_{L_{1,\Psi}} $ of $L_{1,\Psi}(\Omega)$ are $\frac{1}{q_\varphi}$ and $\frac{1}{p_\varphi}$, respectively (see \cite[Theorem 4.2]{Ma1}).  Using \eqref{eq:weak index-strong index}, we get that $1 < p_{L_{1,\Psi}} \le q_{L_{1,\Psi}} <\8$.  If $\tau(1)<\8$, then  by \cite[Theorem 5]{BM}, we obtain the desired result. If $\tau(1)=\8$.  Choose that $1<p<p_{L_{1,\Psi}} \le p_{L_{1,\Psi}} <q<\8$. Since there is a bounded projection operator $P$  from $ L_{p}(\mathcal{M})$
onto $ H_{p}(\mathcal{A})$ and from $ L_{q}(\mathcal{M})$
onto $ H_{q}(\mathcal{A})$ (see \cite[Theorem 4.2]{Be1}),
by   Theorem 3.4 in \cite{DDP1}, we know that $P$ is a bounded projection from $L_{1,\Psi}(\mathcal{M})$
onto $H_{1,\Psi}(\A)$.
\end{proof}

\begin{theorem}\label{thm:dual-orlicz-hardy}
Let  $\Phi$ $1< a_{\Phi} \le b_{\Phi}<\8$.  Then
\beq\label{eq:dual-weak hardy spaces}
H_{\Phi,\8}(\A)^*=H_{1,\Psi}(\A)\oplus S_0|_{H_{\Phi,\8}(\A)}\oplus S_\8|_{H_{\Phi,\8}(\A)}.
\eeq
\end{theorem}
\begin{proof} It is clear that
\be
H_{1,\Psi}(\A)\oplus S_0|_{H_{\Phi,\8}(\A)}\oplus S_\8|_{H_{\Phi,\8}(\A)}\subset H_{\Phi,\8}(\A)^*.
\ee
Let $\ell\in H_{\Phi,\8}(\A)^*$. By Hahn-Banach theorem, there is a functional $\tilde{\ell}\in L_{\Phi,\8}(\M)^*$ such that $\ell=\tilde{\ell}|_{H_{\Phi,\8}(\A)}$. Using \eqref{eq:dual-weak-orlicz}, we get that
\be
\tilde{\ell}(x)=\tau(xy^*)+\tilde{\ell_1}(x)+\tilde{\ell_2}(x),\quad \forall x\in L_{\Phi,\8}(\M),
\ee
where $y\in L_{1,\Psi}(\M) $, $\tilde{\ell_1}\in S_0$ and $\tilde{\ell_2}\in S_\8$. Using Proposition \ref{pro:weak-hardy space}, we obtain that there exist $h\in H_{1,\Psi}(\A)$ and $z\in H_{1,\Psi}(\A)_0$ such that $y=h+z^*$. Hence,
\be
\tau(ay^*)=\tau(ah^*)+\tau(az)=\tau(ah^*),\quad a\in H_{\Phi,\8}(\A).
\ee
Therefore, $\ell=h+\tilde{\ell_1}|_{H_{\Phi,\8}(\A)}+\tilde{\ell_2}|_{H_{\Phi,\8}(\A)}$. From this follows \eqref{eq:dual-weak hardy spaces}.
\end{proof}

If $\Phi(t)=\frac{t^p}{p}\;(1<p<\8)$, then $\Psi(t)=\frac{t^q}{q}$ and  $\varphi(t)=\frac{1}{q^{1/q}}t^{1/q}$, where $\frac{1}{p}+\frac{1}{q}=1$.
Hence,
\be
N_0(x)=\limsup_{t\rightarrow0}q^{1/q}\frac{1}{t^{1/q}}\int_0^t\mu_s(x)ds=q^{1/q}\limsup_{t\rightarrow0}t^{1/p-1}\int_0^t\mu_s(x)ds,
\ee
and
\be N_\8(x)=\limsup_{t\rightarrow\8}q^{1/q}\frac{1}{t^{1/q}}\int_0^t\mu_s(x)ds=q^{1/q}\limsup_{t\rightarrow\8}t^{1/p-1}\int_0^t\mu_s(x)ds.
 \ee

\begin{corollary}\label{thm:dual-orlicz-hardy}
Let  $1< p<\8$.  Then
\beq\label{eq:dual-weak hardy spaces}
 H_{p,\8}(\A)^*=H_{1,q}(\A)\oplus S_0|_{H_{p,\8}(\A)}\oplus S_\8|_{H_{p,\8}(\A)}.
\eeq
\end{corollary}

\subsection*{Acknowledgement} T. B. Bekjan and M. Raikhan are partially supported by  project AP09259802 of the Science Committee of Ministry of Education and Science of the Republic of Kazakhstan.

\end{document}